\documentclass[12pt]{amsart}
\usepackage{graphicx}
\usepackage{amsmath}
\usepackage{amsfonts}
\usepackage{amssymb}

\usepackage{amsxtra}
\usepackage{amssymb,latexsym}
\usepackage[mathcal]{eucal}
\usepackage{amscd}

\newtheorem{lemma}{Lemma}
\newtheorem{theorem}{Theorem}
\newtheorem{proposition}{Proposition}
\newtheorem{remark}{Remark}

\newcommand{\B}{\mathcal{B}}

\newcommand{\PP}{\mathcal{P}}
\newcommand{\R}{\mathcal{R}}
\newcommand{\Q}{\mathcal{Q}}

\newcommand{\G}{\mathcal{G}}
\newcommand{\bbZ}{\mathbb{Z}}
\newcommand{\bbN}{\mathbb{N}}

\newcommand{\PPP}{\boldsymbol{\mathcal{P}}}

\newcommand{\XBT}{(X, \mathcal{B}, \mu, T)}
\newcommand{\XBG}{(X, \mathcal{B}, \mu, T_g)}
\newcommand{\YCS}{(Y, \mathcal{C}, \nu, S)}

\begin{document}

\bibliographystyle{plain}

\title{Some generic properties of processes  }

\author{ Benjamin Weiss}

\address {Institute of Mathematics\\
 Hebrew University of Jerusalem\\
Jerusalem\\
 Israel}
\email{weiss@math.huji.ac.il}

\setcounter{page}{1}
\date{}

\maketitle
\begin{abstract}
 For a given ergodic measure preserving transformation T of a standard measure
space $(X, \mathcal{B}, \mu)$ each finite labelled partition defines an ergodic stationary process. There
is a complete metric on the space of partitions which is separable. Various
generic properties of these processes will be given. For example: $\\$
    1. The generic partition defines a process that is not Rosenblatt mixing. $\\$
    2. If  T is a K-automorphism that is not Bernoulli then the generic partition is also K but
not Bernoulli.

    Extensions to the relative setting and to actions of amenable groups will also be discussed.
\end{abstract}

\section{Introduction}

  The study of generic properties of measure preserving transformations goes back to the early days of ergodic theory, and the first results are clearly
  exposed in the chapter
  on Category in the classic book of Paul Halmos \cite{H}. More recent results can be found, for example in \cite{GTW2}, \cite{AGTW} and \cite{Sc}. In this note I shall
  take up the study of generic properties of the class of finite valued processes that are defined by one fixed ergodic measure preserving system $\XBT$. Such processes are defined
  by labelled partitions of the measure space.

  With the metric on sets defined by the measure of the symmetric difference between two sets, one defines a metric
  on the space of labelled partitions with the same number of elements with respect to which the space becomes a Polish space. The metric that this induces on the processes
  defined by these partitions is the well-known d-bar metric. My study of these generic properties was motivated by the open problem: Does Rosenblatt mixing imply Bernoulli (for the definition see
  the next section). Instead of an explicit construction I thought that perhaps this can be true generically. Instead I found that the opposite is the case, namely for any ergodic system
  the generic process is not Rosenblatt mixing. This shows that the converse implication fails generically as explained in the next section.

    In general, any process defined by a partition defines a factor system. The standard mixing properties descend to factors. On the other hand a system can fail
    to have a mixing property while it might have many factors that do. For example any K-system has many Bernoulli factors even if the system itself is not Bernoulli.
    However, we will see in Section 3 that the generic partition of a non Bernoulli system defines a non Bernoulli process.

    Continuing this study revealed that
  a similar result holds for strong, mild and weak mixing as detailed in $\S 4$. In $\S 5$ these results will be extended to the 
       to the relative setting, and in $\S 6$  to
  actions of countable amenable groups.

   Perhaps a more basic property of a finite partition is that of being a generator.
  The results in section 4 show that for a zero entropy ergodic transformation  the subsets $A$ such that the partition $\{A, X \setminus A \}$ is a generator, is generic. In positive entropy
  the continuity of the entropy shows that a similar statement cannot hold. In $\S 7$ I will give two kinds of extensions of this result to positive entropy.
  In the first, one fixes an arbitrary factor $\B_0$ of full entropy and then  a generic partition together with $\B_0$ generates. In the second, one restricts to the partitions
  that have full entropy, which is a closed set, and in this closed set, the generators form a generic set. These results are analogous to the fact that for positive entropy systems the generic extension
  is isomorphic to the to the base, see \cite{AGTW}.
      As a general reference for the standard facts that I will make use of see \cite{R} and \cite{Sh}.
\section*{Acknowledgement}
  I would like to thank Jon Aaronson, Eli Glasner and Jean-Paul Thouvenot for the helpful discussions and correspondences that we had during
  the course of my work on this paper.

\section{Rosenblatt mixing}

   A stationary stochastic process $\{ X_n \}$ satisfies the Rosenblatt mixing condition if the dependence coefficient
   $\alpha(n)$ defined by

   $$\alpha(n) = \textrm{sup} \{| \mathbb{P}(A \cap B)- \mathbb{P}(A)\mathbb{P}(B)| : A \in \mathcal{F}\{ X_i: i < 0 \} ; \hspace{0.2cm} B \in \mathcal{F}\{ X_i: i \geq n \} $$
    tends to zero as $n$ tends to infinity.

   Here $\mathcal{F}\{ X_i: i \leq 0 \}$ denotes the $\sigma$-algebra generated by the random variables $\{ X_i: i < 0 \}$ and similarly for
   $\mathcal{F}\{ X_i: i \geq n \}$.

   Such processes are also called $\alpha$-mixing. Nathaniel Martin showed in \cite{M} that if the decay is rapid enough then the process is isomorphic to a Bernoulli shift.
   Somewhat earlier, Meir Smorodinsky in \cite{Sm} exhibited an example of a process that was isomorphic to a Bernoulli shift but was not Rosenblatt mixing.
   His proof relied on an unpublished result of Harry Kesten. In \cite{B} Bradley gave a detailed proof of Kesten's result. In this section we will
    show that non-Rosenblatt mixing processes are in fact generic in any ergodic system, in particular in Bernoulli shifts. Since factors of Bernoulli shifts are also 
    Bernoulli this gives a plethora of examples of non-Rosenblatt mixing transformations that are Bernoulli. 

   Let $\XBT$ be an aperiodic ergodic process. For simplicity we will work in the space of two valued processes but the argument will work just as well for
    finite valued processes. Any function $f:X \rightarrow \{0, 1 \}$ defines a stationary process by setting $\{ X_n(x) = f(T^nx) \}$. With the
   $L_1$-metric these functions form a complete separable metric space $Z$. We shall show that the set of  functions that define non-Rosenblatt mixing processes
   contains a dense $G_{\delta}$ subset of this space.

   It will be convenient to identify elements  $f \in Z$ with the partition
   $\PP = \{ P_0, P_1 \}$
    they define where $P_i = f^{-1}(i)$. We will use the notation
   $ \PP_k^l = \bigvee_{i=k}^l T^{-i}\PP$. Define now the subsets:

   $$U_n = \{ f \in Z: \exists \hspace{0.2cm}  k \in \bbN , \hspace{0.2cm}  \textrm{and sets} \hspace{0.2cm} C \in \PP_{-k}^0 , \hspace{0.2cm} D \in \PP_n^{n+k}$$
   $$ \textrm{such that} \hspace{1cm} |\mu(C \cap D) - \mu(C) \mu(D)| > \frac{1}{10} \}.$$

   These sets are clearly open and if $ f \in \bigcap U_n$ then clearly the process that $f$ defines is not Rosenblatt mixing. It thus remains to show, and this is of course the more difficult
   part of the proof, that each $U_n$ is dense.
     For this we will use a version of the Rokhlin lemma which is perhaps not so widely known
     and so we will sketch a proof in the ergodic case.

     \begin{lemma}
      If $\XBT$ is aperiodic and ergodic then for any $N$ one can find a measurable set $B$ such that
    the sets $ \{ T^iB : 0 \leq i < N \}$ are disjoint while $\mu ( \bigcup_{i=0}^N T^iB)  = 1 $.
     \end{lemma}
     \begin{proof}
  By the aperiodicity we can find a set $A$ of positive measure so that
    the sets $ \{ T^iA : 0 \leq i < N^2  \}$ are pairwise disjoint. Let $r_A(x)$ be the time of the first return of a point $x \in A$ to the set $A$. By assumption
    $r_A(x) \geq N^2 $ and we partition $A$ into $A_k = \{ x \in A ; r_A(x) = k \}$. Each column above an $A_k$ can be partitioned into disjoint consecutive blocks of size
    $N$ and $N+1$ since the height is at least $N^2$. Indeed writing the height $H \geq N^2$ as $H = qN + r$ with $ r < N$ and $q \geq N-1$ we can use $r$ blocks of size $N+1$ and then
     continue with blocks of size $N$ until we exhaust the column of height $H$. By ergodicity these columns exhaust the measure of the space and we define $B$  to be the first level
    in each  of these blocks. Clearly $B$ will have the desired properties. $\Box$
    \end{proof}

This decomposition of the space is called a $(B,N)$- \textbf{Kakutani-Rohlin} tower, or briefly a
$(B,N)$-KR tower.

   We return now to the proof that each $U_n$ is dense.
 Suppose that we are given a function $f \in Z$ with a corresponding partition $\PP$, a $U_n$ and  $\epsilon >0$. Choose
$N > n$ and also such that $\frac{1}{N} < \frac{\epsilon}{10}$. let $B$ be the base of an $(B, 3N^2)$-tower. We modify $\PP$ to $\Q$ in two steps.

\textbf{Step 1}. Put in $Q_0$ all points in the first $N$ levels of the tower, i.e. $\cup_{i=0}^{N-1} T^iB \subset Q_0$, then put $T^NB$ in $Q_1$ and the next $N$ levels again in $Q_0$,
 $\cup_{i=N+1}^{2N} T^iB \subset Q_0$.

\textbf{Step 2}. All levels of height $k\frac{N}{2}$ with $k >4$ are put in $Q_1$.

  On all the other levels we put $ x \in Q_i \hspace{0.2cm} \textrm{if and only if} \hspace{0.2cm} x \in P_i$ .

  Clearly the resulting partition $\Q$ differs from $\PP$ by at most $\epsilon$. It remains to show that $\Q \in U_n$. The purpose of the modification that we made was to ensure that
  the base $B$ is now measurable with respect to $\bigvee_{i=a}^b T^{-i}\Q$ for all $a < b$ such that $b - a > 3N^2 + 2N$. This is because the string $s=0^N10^N$ occurs
  along the $\Q$-name of  $\mu$-a.e. point uniquely with gaps either $3N^2$ or $3N^2 + 1$ apart. To see that indeed $\Q \in U_n$ in the definition of $U_n$ take $ k= N^2$ and define the sets
  $C$ and $D$ as follows:

   $$C = \{ x : T^i(x) \in B \hspace{0.2cm} \textrm{for some} \hspace{0.2cm} i \in [-N^2 , -2N) \}$$
   and
    $$D = \{ x : T^i(x) \in B \hspace{0.2cm} \textrm{for some} \hspace{0.2cm} i \in [n+2N +N^2] \}.$$

    Each of these sets has measure slightly less than $\frac{1}{3}$ while since $n < N$ their intersection is empty. It follows that $\Q \in U_n$ and this concludes the proof of the following
    theorem:

    \begin{theorem} In any ergodic system $\XBT$ for a generic set of two valued functions the process that they define is not Rosenblatt mixing.
    \end{theorem}

    A similar result is valid for finite valued processes with a similar proof.

    \section{K-processes that are not Bernoulli}

    Suppose now that $\XBT$ is  K-system. By Sinai's theorem there are many partitions that define Bernoulli shifts. Recall that we call a process  \textbf{Bernoulli} if it is
    isomorphic to an independent process. Don Ornstein showed that all factors of a Bernoulli process are Bernoulli. The fact that a d-bar limit of Bernoulli processes is also Bernoulli
    implies that the collection of Bernoulli processes in any system is closed in the space of partitions. If $\XBT$ is not isomorphic to a Bernoulli system then we shall show that
    those partitions that define a non-Bernoulli process are generic. To prove this, what we have to show is that the non-Bernoulli partitions are dense since we already know that
    the set of partitions generating non-Bernoulli factors is open.  We will need the weak-Pinsker
    property that was established by Tim Austin \cite{A} a few years ago.
      The weak-Pinsker property asserts that if $\XBT$ is a positive entropy system then for any $\epsilon > 0$ there are
      independent factors $(X, \B_i, \mu, T)$  for $i=0,1$ such that:
      (i)
      the entropy of $(X, \B_0, \mu, T)$ is less than $\epsilon$ ; (ii)  $(X, \B_1, \mu, T)$ is Bernoulli; and (iii) $\B = \B_0 \bigvee \B_1$.

      It follows that if $\XBT$ is non-Bernoulli then $(X, \B_0, \mu, T)$ is necessarily non-Bernoulli.

       \begin{theorem}\label{nonB} If the ergodic system $\XBT$ has positive entropy and is not Bernoulli then the non-Bernoulli partitions form a dense open set
       in the space of partitions.
       \end{theorem}

      \begin{proof}
               We are given a partition $\PP$ of $X$,
         $\epsilon > 0$ and our goal is to find a partition $\PP'$ that is within $\epsilon$ of $\PP$ and is not Bernoulli. We  apply Austin's theorem to
        find a non-Bernoulli factor $B_0$  with entropy  $h(X, \B_0, \mu, T) = h < \frac{\epsilon}{10}$.

         By a result of Manfred Denker \cite{D} there is a finite generating partition
          $\Q$ for $\B_0$ such that the subshift defined by $\Q$ is strictly ergodic and hence has
           topological entropy equal to $h$. This means that if we denote by  $\mathcal{C} \subset \bigvee_{i=0}^{n-1} T^{-i}\Q $ those atoms that have positive
           measure then for $n$ sufficiently large $|\mathcal{C}|  < 2 ^{\frac{\epsilon}{5}n}$.

        Fix now some $N$ such that $\frac{1}{N} < \frac{\epsilon}{10}$ and an  $n$ so that $\frac{N}{n} < \frac{\epsilon}{20}$.
        Let now $B_0, B_1$ be the bases of an $\{n, n+1\}$-KR tower.
        Partition the base $B_0$ by $  \bigvee_{i=0}^{n-1} T^{-i}\Q $ and $B_1$ by   $\bigvee_{i=0}^{n} T^{-i}\Q$. If $A \subset B_0$  is an atom of this partition
        we modify $\PP$ on the levels $\{ T^iA : 0 \leq i \leq n-1 \}$ as follows. The first $N$ levels are put in $P_0'$, the next level is put in $P_1'$, the next  $N$ levels  are put in $P_0'$ and then
        every $\frac{N}{2}$-th level is put in $P_1'$. Now let $\pi:\mathcal{C} \rightarrow \{0, 1 \}^{\frac{\epsilon}{5}n}$ be a one to one map and use this map to relabel
        the next $\frac{\epsilon}{5} n$ free levels (i.e. levels that haven't been modified in the first part of the argument) to encode the $\Q-n$-name of the atom $A$.
        For all other levels $T^iA$ we set $P_j' =P_j$.
        The columns above the base $B_1$ are
        treated similarly except that the $N$ is replaced by $2N$ so that we can distinguish between $B_0$ and $B_1$.

          Clearly  $\PP'$ differs from $\PP$ by less than $\delta$ and given $\PP'$-name of a point one can reconstruct uniquely the $\Q$-name so that the process defined by $\PP'$ is not
          Bernoulli. $\Box$
       \end{proof}

 % \begin{remark}In case the system $\XBT$ has a non-trivial Pinsker factor, i.e. a maximal factor with zero entropy, then the proof shows that the
  %     partitions $\PP$ that generate a factor $\bigvee_{-\infty}^{\infty}T^{-i}\PP$ that contains the Pinsker factor of $\XBT$ are dense. While it is not clear that this set is open
   %    it is easily seen to be a $G_{\delta}$ and thus we can say that a generic set of partitions contain the Pinsker factor. In particular, if the
    %   system has zero entropy then a generic set of partitions are generators.
     %  \end{remark}

\section{Mixing properties}

     Recall the basic mixing properties of an ergodic system $\XBT$: weak mixing (WM), mild mixing (MM), strong mixing (SM)and  Kolmogorov processes (or, equivalently
     completely positive entropy (CPE) systems). It is straightforward to see that each of these properties
     descends to factors. It follows that if a system has one of these properties then all partitions of the space define processes that have the same property. What we shall show is, that if a system
     fails to have one of these properties then for an open dense set of partitions the process defined by the partition fails to have that property.

     The first thing to observe is that each of these mixing properties is closed in the d-bar metric. For K-processes this is exactly  (\cite{R} Exercise 7.5).
     In our situation, where all the processes are coming from partitions of a fixed system this is fairly easy to see and here is a sketch of the proof.

      We are given a sequence of finite partitions $\{ \PP_n \}$ that converge to $\PP$ such that each one defines a K-process. By lumping atoms, which preserves the K-property
      we may assume that the number of atoms of the $\PP_n$ is the same as that of $\PP$. We now use the Rokhlin metric on partitions given by $\rho( \Q , \R ) = H(\Q|\R) + H(\R |\Q)$.
      If $C$ is any set in the $\sigma$-algebra generated by $\PP$ our task is to show that the entropy of the partition  $ \mathcal{C} =\{C, X \setminus C \}$ is positive. Now for a small enough
      $\epsilon$ we find a $C_0 \subset \bigvee_{i= -N}^{+N} T^i\PP$ such that,
       setting $\mathcal{C}_0 = \{C_0 , X \setminus C_0 \}$
       $\rho( \mathcal{C} ,\mathcal{C}_0) < \epsilon$. Since the $\PP_n$ converge to $\PP$ if $n$ is sufficiently large
      $\rho( \bigvee_{i= -N}^{+N} T^i\PP_n , \bigvee_{i= -N}^{+N} T^i\PP) < \epsilon$.

      Letting $C_n$ be the subset of $\bigvee_{i= -N}^{+N} T^i\PP_n$  that corresponds to $C_0$ it follows that, setting $\mathcal{C}_n = \{C_n , X \setminus C_n \}$ we have
       $\rho( \mathcal{C}_n ,\mathcal{C}_0) < \epsilon$.

      Now the fact that $\PP_n$ defines a K-process implies that the process entropy $h( T^k ,\mathcal{C}_n)$ tends to $H(\mathcal{C}_n)$ as $k$ tends to infinity.
      Since process entropy is a Lipschitz function in the Rokhlin metric if $\epsilon$ is sufficiently  small this implies that process entropy of
      $\mathcal{C}$ is positive.

     In \cite{Sh}
     it is shown that both ergodicity (Theorem I.9.15) and mixing (Theorem I.9.17) are closed in the d-bar metric. That weak mixing is closed follows from the characterization
     of weak mixing as the ergodicity of the product system. We have not found a reference in the literature to the fact that mild mixing is closed in d-bar. This is not so obvious from the usual
     definition of mild mixing as the absence  of non-trivial rigid factor. However, there is an equivalent definition using the notion of IP*-convergence which can be found in (\cite{F} Proposition  9.22).
     With this definition it is quite straightforward to extend the proof of the closure of mixing in \cite{Sh} to mild mixing. Jean-Paul Thouvenot pointed out that yet another proof can be based on the
     characterization of mild-mixing as being disjoint from all rigid transformations. This characterization can be found in (\cite{G1} Corollary 8.16).

     The next observation is that the failure of any one of these properties is witnessed by the existence of some special factor that has zero entropy. For weak mixing this is the Kronecker factor.
     For mild mixing this is a rigid factor which necessarily has zero entropy. For non-K processes this is the Pinsker factor. These are all well known. For non-mixing this is a
     result of Fran\c{c}ois Parreau \cite{P}. He showed that if a system is not mixing then it has a factor which is disjoint from all mixing transformations. As such it must have zero entropy.
     An exposition of Parreau's theorem was given by Eli Glasner  (see \cite{G2}).

     Our result will now follow from the next two propositions.
         \begin{proposition}\label{Gdelta}
         If $\Q$ is any fixed partition then the set of partitions
             $\PP$ such that $ \Q \subset \bigvee_{-\infty < i <+\infty}T^i\PP$ is a  $G_{\delta}$.
         \end{proposition}
         \begin{proof}
         Define for  fixed $k, n$ the set:

             $$ U(k,n) = \{ \PP : \hspace{0.2cm} \exists \Q_0 \subset \bigvee_{i=-n}^{n}T^i\PP \hspace{0.2cm} \textrm{such that} \hspace{0.2cm} d(\Q_0, \Q) < \frac{1}{k} \}. $$

           These sets are clearly open and if $\PP \in \bigcap_{1 < k < \infty} \bigcup_n U(k,n)$ then $\PP$ generates a factor containing $\Q$. $\Box$
         \end{proof}

         \begin{proposition}\label{modify}  Let $\XBT$ be an ergodic system and $\PP = \{P_0, P_1, ...P_a \}$ be a partition of $X$. Suppose that $\B_0$ is  a factor sub-$\sigma$-algebra. If the factor  has zero entropy then
         there exists a finite generating partition $\Q$ for the factor $\B_0$  such that
          for any  $0 < \epsilon$, there is a partition $\PP'$ satisfying:
          \begin{enumerate}
          \item $d(\PP, \PP') < \epsilon$, and
          \item $\Q \subset \bigvee_{-\infty}^{\infty}T^{-i}\PP'$.
       \end{enumerate}
          \end{proposition}

     \begin{proof}
       The proof is essentially the same as the proof of Theorem \ref{nonB}. The partition $\Q$ in the assertion is the one we get from Denker's result.
     $\Box$
          \end{proof}

        %  \begin{remark}\label{positive}A more careful argument will show that the assumption of zero entropy for the factor can be replaced by the assumption that the factor $\B_0$
         % has entropy $h < \epsilon$.
          %\end{remark}

          The following theorem summarizes what we have done.

          \begin{theorem}\label{general} Let $\XBT$ be an ergodic system.
          \begin{enumerate}
          \item If the system is not weakly mixing then the set of partitions $\PP$ such that the process they define is not weakly mixing is
          both dense and open. In addition, the set of partitions that generate a factor containing the Kronecker factor is a dense $G_{\delta}$.

          \item If the system is not mildly mixing then the set of partitions $\PP$ such that the process they define is not mildly mixing is
          both dense and open. In addition, for any rigid factor $\B_0$ the set of partitions that generate a factor containing $\B_0$ is a dense  $G_{\delta}$.

          \item If the system is not  mixing then the set of partitions $\PP$ such that the process they define is not mixing is
          both dense and open. In addition, for any Parreau factor $\B_0$ the set of partitions that generate a factor containing $\B_0$ is a dense  $G_{\delta}$.

          \item If the system is not Kolmogorov then the set of partitions $\PP$ such that the process they define is not Kolmogorov is
          both dense and open. In addition, the set of partitions that generate a factor containing the Pinsker algebra is a  dense  $G_{\delta}$.
          \end{enumerate}
          \end{theorem}
          \begin{proof}
            The fact that the various non-mixing partitions form open sets was discussed above. Proposition \ref{modify} shows that they are dense. For the second part we use in addition
            Proposition \ref{Gdelta}.
             $\Box$
          \end{proof}

        \begin{remark}\label{generators} It follows from the theorem that for a zero entropy system for a generic collection of subsets $A$ the partition $\{A, X \setminus A \}$ is a generator.
        For positive entropy this of course cannot be true. However if $\XBT$ has finite positive entropy $h$ and we fix any full entropy factor defined by an invariant $\sigma$-algebra $\B_0$ then a generic partition $\PP$  together with $\B_0$  generates $\B$.  (see the next section).

        \end{remark}

\section{The relative theory}

    When the system $\XBT$ has a factor $\YCS$, i.e there is a measurable mapping $\pi:X \rightarrow Y$ such that $\pi \circ \mu = \nu$ and $\pi T = S \pi $ there
    are well studied notions of relative weak-mixing (RWM), relative Kolmogorov (RK) and relative Bernouli (RB). It will be convenient to identify the factor
    with $\pi^{-1}\mathcal{C}$ which is a $T$-invariant sub-$\sigma$-algebra of $B$, and henceforth $\mathcal{C}$ will be used to denote this.  For each of these properties the analogue of
    our previous results would be that if the extension $\XBT$ fails to have that property relative to the factor $\YCS$ then the process defined by a generic partition $\PP$
    of $X$ also fails to have that property relative to $\YCS$. What we mean by that is that system formed by joining $\bigvee_{ -\infty <i< +\infty}T^i\PP$ to $\mathcal{C}$
    the resulting extension of $\YCS$ fails to have that property relative to $\YCS$.

      As before, the first thing to check would be that the partitions that define processes that do satisfy one of these relative properties form a closed subset
      in the space of partitions. For relative weak mixing this follows easily from the definition RWM as the ergodicity of the relative product of $\XBT$ with itself over the factor $\YCS$.
       For relative Bernoulli this follows from Prop. 7 in Jean-Paul's foundational work on the relative Bernoulli theory \cite{T}.

       For relative K a brief proof can be given based on the following characterization of relative K which can be found in \cite{GTW1}.
        An extension $\XBT$ of $\YCS$ is relatively K if and only if it is relatively disjoint (over $\YCS$ ) from every extension that has
        relative entropy equal to zero.

    What is needed in order to prove that the relevant open set is dense is a relative version of Proposition \ref{modify}. However, since we are not aware of a relative analogue
    of the result of Denker that was used in the proof of that proposition we will have to modify
     the proof.

      \begin{proposition}\label{relmodify}  Let $\XBT$ be an ergodic system which is an extension of $\YCS$ by the factor map $\pi$, and let $\PP = \{P_0, P_1, ...P_a \}$ be a partition of $X$.
      Suppose that $\Q$ is a partition of $X$ such that the process defined by $\Q$  has zero entropy relative to $\pi^{-1}(\mathcal{C})$. Then
          for any  $0 < \epsilon$ and $0 < \delta$ , there are partitions $\PP'$ and $\Q_0$ satisfying:
          \begin{enumerate}
          \item $d(\PP, \PP') < \epsilon$,
          \item $d(\Q_0, \Q) < \delta$, and
          \item $\Q_0 \subset \pi^{-1}(\mathcal{C}) \bigvee_{-\infty}^{\infty}T^{-i}\PP'$.
       \end{enumerate}
          \end{proposition}

     \begin{proof}

          The proof is similar to the proof of Proposition \ref{modify} except that instead of the usual Shannon-McMillan theorem we need to use the relative
          Shannon-McMillan theorem (see \cite{KW} for an appropriate version) and carry out the construction on each fiber $\pi^{-1}(y)$. In addition since we don't have any uniformity for the partition
          $\Q$ we have another error since the fact that $\Q$ has relative zero entropy only means that on typical fibers we can cover only most of the fiber by an
          exponentially small number of atoms from $\bigvee_{i=0}^{n-1}T^{-i}\Q$.

          We leave this routine modification  to the reader.    $\Box$
          \end{proof}

    \begin{remark}\label{relpositive}  A slightly more careful argument will show that the assumption of relative zero entropy for $\Q$  can be replaced by the assumption that its
    relative entropy $h$ satisfies $h < \epsilon$.
          \end{remark}
      We can now state the main result of this section.

       \begin{theorem}\label{relgeneral} Let $\XBT$ be an ergodic system with a factor defined by $\mathcal{C} \subset \B$. \\

          \textbf{(i) }If the system is not  weakly mixing relative to $\mathcal{C}$ then the set of partitions $\PP$ such that the process they define is not weakly mixing
          relative to $\mathcal{C}$  is
          both dense and open. \\

          \textbf{(ii)} If the system is not Kolmogorov relative to $\mathcal{C}$ then the set of partitions $\PP$ such that the process they define is not Kolmogorov
          relative to $\mathcal{C}$ is
          both dense and open. \\

           \textbf{(iii)}If the system is not Bernoulli relative to $\mathcal{C}$ then the set of partitions $\PP$ such that the process they define is not Bernoulli
          relative to $\mathcal{C}$ is both dense and open.

          \end{theorem}
          \begin{proof} \\

          \textbf{(i)} If the extension is not relatively weakly mixing there is a non-trivial intermediate factor which is a distal extension of $\YCS$. This distal extension has
          zero entropy relative to $\YCS$. By the relative finite generator theorem (see \cite{KW} and \cite{DP}) there is a finite partition $\Q$ that generates this distal extension
          over $\mathcal{C}$ and has zero entropy relative to $\mathcal{C}$. By the openness of the set of partitions that generate non RWM extensions, there
          is some positive $\delta$ such that if $\Q_0$ is within $\delta$ of $\Q$ then $\Q_0$ also generates a non RWM extension. Given an arbitrary partition $\PP$
          and an arbitrary $\epsilon$ we can apply Proposition \ref{relmodify} to find a $\PP'$ within $\epsilon$ of $\PP$ and a $\Q_0$ within $\delta$ of $\Q$ such that
          the extension generated by $\PP'$ contains $\Q_0$ and is therefore non RWM. This shows that the set is also dense as required. \\

          \textbf{(ii)} This proof is the same as case (i) because if the extension is not Kolmogorov there is a non-trivial intermediate factor with zero entropy. \\

          \textbf{(iii)} In this case we must begin with the partition $\PP$ and a positive $\epsilon$. By the relative weak-Pinsker property ( see \cite{A}) there is an
          intermediate extension $\mathcal{C'}$ with relative entropy $ h < \epsilon$ such that $\XBT$ is relatively Bernoulli over  $\mathcal{C'}$. We can now apply
          Remark \ref{relpositive} and conclude the proof as in case (i). $\Box$
         \end{proof}

\section{Amenable groups}

     There is a well developed theory of entropy and Bernoulli systems for arbitrary countable amenable groups $G$ (see \cite{OW} )and Theorem \ref{nonB} makes good sense for any such group. In this section we shall extend
     Theorem \ref{nonB} to this setting. For the first part, namely that the partitions that define a Bernoulli process form a closed set, we use the fact that the Bernoulli processes are characterized by
     being Finitely Determined, just as in the case of $\bbZ$. It follows from  (\cite{OW} III. Theorem 4) that these processes are closed under d-bar limits and thus the non-Bernoulli
     partitions form an open set.

      In the proof of the theorem for $\bbZ$ we used strictly ergodic partitions. While there is an extension of this result to
     amenable groups by Alain Rosenthal \cite{Ro} we prefer to give a proof which does not rely on this. For clarity we will first describe this alternative proof for $\bbZ$.

         Once again we start with a partition $\PP = \{P_0, P_1, ...P_a \}$ and an $\epsilon >0$.
         Next by the weak Pinsker property we can find some non Bernoulli factor that has entropy  equal to $h < \frac{\epsilon}{2}$. By Krieger's generator theorem
          this factor has a finite generating partition $\Q$, and this $\Q$ is of course non Bernoulli.
         Since the non Bernoulli partitions are open there is some $\delta > 0$ such that any partition that is within $\delta$ of $\Q$
         is also non Bernoulli. We will now show how to modify $\PP$ by less than $\epsilon$ to obtain a $\PP'$ in such a way that the process generated by $\PP'$ will have as a factor
         a partition $\Q'$ with $d(\Q, \Q') < \delta$. This will show that $\PP'$ is non Bernoulli and $d(\PP, \PP') < \epsilon$ which concludes the proof.

           By the Shannon-McMillan theorem if $n$ is sufficiently large there is a collection $\G$ of atoms of $\bigvee_{i=0}^{n-1} T^{-i}\Q$ such that:
          (i) $\mu(\cup \G ) > 1-\frac{\delta}{10}$ ; and (ii) $|\G| < 2^{ \frac{\epsilon}{2}n}$. Choose now some $N$ such that $\frac{1}{N} < \frac{\epsilon}{10}$ and
          demand also that $n$ is large enough so that $\frac{N}{n} < \frac{\epsilon}{10}$. Apply now the strong Rokhlin lemma to find a set $B$ that is independent
          of $\cup \G$ and such that:(i) the sets $\{T^iB\}_{i=0}^{n-1}$ are disjoint; and (ii) $\mu (\cup_{i=0}^{n-1}T^{-i}B) > 1 - \frac{\delta}{10}$.

             In order to be able to identify the base $B$ we first change $\PP$ on the first $2N+1$ levels of the tower above $B$ by putting all these levels, except for level $N+1$ into the new $P'_0$.  Then
           every $\frac{N}{2}$-th level is put into the new $P'_1$. In addition all points not in $\cup_{i=0}^{n-1}T^{-i}B)$ are also placed in
            $P'_1$. Finally we consider the partition of the tower induced by the partition of the base according to the atoms of $\bigvee_{i=0}^{n-1} T^{-i}\Q$. To the atoms that are
            in $\G$ we can assign distinct elements  $w \in \{0,1\}^{\frac{\epsilon}{2}n}$ and use these elements to place those levels beyond $2N$ that weren't already placed in $P'_1$
       into $P'_0$ or $P'_1$ according to the labels in $w$. For the atoms that are not in $\cup\G$ we reserve a single element $w$ that is not used in the correspondence with $\G$.

              For all other points in the tower $P'_0 = P_0$ and $P'_1=P_1$. It is clear that $d(\PP, \PP') < \epsilon$ and the process defined by $\PP'$ encodes a partition $\Q'$ that is within
             $\delta$ of $\Q$ so that $\PP'$ is non Bernoulli.

           The argument above establishes the following proposition;

           \begin{proposition}\label{modify-new}  Let $\XBT$ be an ergodic system and $\PP$, $\Q$ a pair of partitions. If the entropy $h$ of the process defined by $\Q$ satisfies
           $h < \epsilon$ then for any $\delta > 0$ there are $\PP'$ , $\Q'$ such that:
          \begin{enumerate}
          \item $d(\PP, \PP') < \epsilon$,
          \item $\Q' \subset \bigvee_{-\infty}^{\infty}T^{-i}\PP'$, and
          \item $d (\Q, \Q') < \delta$.
          \end{enumerate}
          \end{proposition}

             As we have seen the proof of Theorem \ref{nonB} is quite straightforward assuming the above proposition. Now using the theory developed in \cite{OW} it is a routine matter to extend
             the proposition to arbitrary countable amenable groups. There are two basic ingredients that were used. The Shannon-McMillan theorem and the strong Rokhlin lemma. The Shannon-McMillan
             theorem is valid for an arbitrary sequence of F{\o}lner sets and requires no special adaptation. For the strong Rokhlin lemma there is a version using many F{\o}lner sets
             (see Theorem 6 on pp. 60-1 in \cite{OW}). We won't go into much detail here, with the exception of what needs to be done to encode the bases of the Rokhlin towers for the
           F{\o}lner sets.

            If we don't insist on limiting the number of sets in our partitions then this encoding can be easily done using extra symbols. However,
              if we want to restrict the number of elements of the partition, as was done in the proof in the proof of Theorem \ref{nonB} (where in the passage from $\PP$ to $\PP'$ no extra symbol
             was used) a more elaborate argument is required.
           For $\bbZ$ we used two things, a  block marked $0^N10^N$ ,  and then the symbol $1$ at a low density arithmetic progression so that the large $0$-block would be
              uniquely placed.  For a general group $G$ the arithmetic progression is replaced by a $K$-separated and $K^2$-covering set for large finite  symmetric $K \subset G$
              that contains the identity.

              In more detail, we say that a subset $C \subset G$ is \textbf{ K-separated} if the sets $\{ Kc : c \in C \} $  are disjoint. If $K$ is symmetric any maximal K-separated
              set $C$ is a \textbf{$K^2$-covering}, which means that $K^2C = G$. The $K$ will be chosen so that $\frac{1}{|K|}$ is sufficiently small. Then we will let $L$ be a larger
              symmetric set so that $L$ will contain in addition to $K^2$ another disjoint translate of $K^2$. Then if the F{\o}lner set $F$ is sufficiently invariant and $\frac{|K|}{|F|}$
           is sufficiently small we can mark the base of an $F$ tower by placing all levels corresponding to $g \in L$ in $\PP_0'$
         with the exception of the base itself which is put into $\PP_1'$.
          Then on the rest of tower we take a maximal K-separated set and place the levels corresponding to this set in $\PP_1'$.
      In addition all points not in the Rokhlin towers are placed in $\PP_1'$.

         After marking the base of the tower the rest of the proof proceeds exactly as in the case of $\bbZ$.
         Here is the proposition whose proof we have just sketched:

              \begin{proposition}\label{modify-amen}Let $\XBG$ be an ergodic system where $G$ is a countable amenable group, and $\PP$, $\Q$ a pair of partitions. If the entropy $h$ of the process defined by $\Q$
              satisfies
           $h < \epsilon$ then for any $\delta > 0$ there are $\PP'$ , $\Q'$ such that:
          \begin{enumerate}
          \item $d(\PP, \PP') < \epsilon$,
          \item $\Q' \subset \bigvee_{-\infty}^{\infty}T^{-i}\PP'$, and
          \item $d (\Q, \Q') < \delta$.
          \end{enumerate}
          \end{proposition}

             With this we can prove:
             \begin{theorem}\label{nonB-amen} If $G$ is a countable amenable group and the ergodic system $\XBG$ has positive entropy and is not Bernoulli then the non-Bernoulli partitions form a dense open set
       in the space of partitions.
       \end{theorem}

              \begin{proof}

    We first note that, just as is the case for $\bbZ$, for any countable amenable group $G$, the Bernoulli
  processes are closed in the d-bar metric. This implies that the non-Bernoulli partitions form an open set, and our task is to show that this set is dense. Given a partition $\PP$ and an $\epsilon$
  we first apply the weak Pinsker property for amenable groups (see \cite{A}) to find a non Bernoulli factor with entropy  $h  < \epsilon$.  We now need to use the extension of
  Krieger's finite generator theorem to arbitrary countable amenable groups which can be found for example in \cite{DP}. This will give a generating partition $\Q$ for this non Bernoulli factor. Since
  the non Bernoulli partitions are open there is some $\delta > 0$ so that any partition within $\delta$ of $\Q$ is also non Bernoulli. Now we can apply Proposition \ref{modify-amen} and conclude the proof. $\Box$

     \end{proof}

\section{Generators}

  To begin with I will formulate Remark \ref{generators} as formal theorems.

  \begin{theorem}\label{zero entropy} If $\XBT$ is ergodic and has zero entropy then for a generic set of subsets $A \subset X$ the partition $\{A, X \setminus A \}$ is a generator.
  \end{theorem}

    This follows immediately from Propositions \ref{Gdelta} and \ref{modify} in section 4.

  \begin{theorem}\label{rel zero entropy} If the ergodic system $\XBT$ has positive entropy and $\B_0$ is an invariant sub-sigma algebra that has full entropy
  then  for a generic set of subsets $A \subset X$ the partition $ \PP = \{A, X \setminus A \}$ satisfies $\B = \B_0 \bigvee ( \bigvee_{-\infty}^{\infty} T^{-i}\PP) $,
  i.e. $\PP$ is a relative generator.
  \end{theorem}

   This theorem follows easily from Proposition \ref{relmodify} in section 5.

   If we are looking for finite partitions in positive entropy systems that generate the full
   sigma-algebra by themselves then the value of the entropy gives a lower bound on the number of sets in the partition.
   That this is the only restriction was first
    shown by Krieger in \cite{K}. He showed that if the system $\XBT$ has positive entropy $h <$  log $a$ then
    there exists partition a $\PP$ of $X$ with $a$ elements that generate the sigma-algebra $\B$.  As always, it is understood
   that this is modulo $\mu$-null sets. In \cite{BKS} a new proof of this theorem was given by showing that in a certain space of joinings a generic element
   will define a generating partition with $a$ elements. However this space of joinings takes us out of the space of partitions of $X$ with $a$ elements.

   By the continuity of the entropy the space of partitions whose entropy is equal to $h$ is closed.  I mean by this, the entropy of the
   process defined by the partition equals $h$. The next theorem claims that in this space the generators form a generic set. This requires a new proof since to prove the density
   of generators in this set we have to perturb an arbitrary partition with full entropy to get a generator without increasing the number of elements of the partition.

     \begin{theorem} Let $\XBT$ be an ergodic system with entropy equal to $h <$ log $a$, and denote by $\boldsymbol{\mathcal{P}}_a$ the space of partitions into $a$ elements
     that define processes with entropy equal to $h$. For a dense $G_{\delta}$ subset of $\boldsymbol{\mathcal{P}}_a$ the partitions are generators.
    \end{theorem}

    \begin{proof} By the theorem of M. Denker in \cite{D} there exists a partition $\Q$  into $a$ sets that is a generator and furthermore,
           the subshift defined by $\Q$ is strictly ergodic and hence has
           topological entropy equal to $h$. The claim that the set of generating partitions  in $\boldsymbol{\mathcal{P}}_a$ is a $G_{\delta}$ is proved using the same
           proof as that of Proposition \ref{Gdelta} in section 4. We repeat the definition of the relevant open sets:

           $$ U(k,n) = \{ \PP \in \PPP_a : \hspace{0.2cm} \exists \Q_0 \subset \bigvee_{i=-n}^{n}T^i\PP \hspace{0.2cm} \textrm{such that} \hspace{0.2cm} d(\
           Q_0, \Q) < \frac{1}{k} \}. $$

           Each of these sets are clearly open and a partition is a generator if and only if it is contained in $\bigcap_{k=1}^{\infty} \bigcup_{n=1}^{\infty} U(k,n)$. Now comes the more difficult
           part of the proof. We must show that for every fixed $k$ the open set  $ V_k = \bigcup_{n=1}^{\infty}U(k,n)$ is dense in $\PPP_a$. We are given an arbitrary element $\PP \in \PPP_a$
           and some $\epsilon >0$. Our task is to construct $\hat{\PP}$, an $\epsilon$-perturbation of $\PP$, that is a member of $V_k$. In particular it has to have full entropy. To accomplish this we will show that
           the invariant sigma-algebra that it generates already contains $\Q$.

           Fix a small $\delta >0$ which will be specified later. As a first step we replace $\PP$ by a partition $\PP'$ that satisfies:
           \begin{enumerate}
           \item $d(\PP, \PP') < \delta$, and
           \item the subshift defined by $\PP'$ is strictly ergodic.
           \end{enumerate}

           For this fact see \cite{W} or \cite{D}. In the latter paper while this fact is not stated explicitly it does follow from the detailed proof given there.
           The  entropy of $\Q$ relative to $\PP'$ tends to zero as $\delta$ tends to zero. By the relative Shannon-McMillan theorem for any sufficiently large $N$
           we can find a $(B,N)$-KR tower such that the following holds.  

             We first divide the tower into pure columns with respect to the partition $\PP'$.
             For all but an $\frac{\epsilon}{3}$ fraction of these columns when a column is further
             purified by the partition $\Q$ the number of  these sub-columns does not exceed  $e^{\frac{\epsilon}{100}\times N}$. 
             
             In the ensuing discussion, in order not to complicate the notation I have treated the tower
           as though all columns have height $N$. Of course everything that follows should be done separately for the two bases and their columns of heights $N$ and $N+1$.

             Returning to the $\PP'$-pure columns which  have a very small exponential number of $\Q$ refinements, 
             we will change the partition $\PP'$ on the first $\frac{\epsilon}{3}\times N$ levels in order to include a coding of the 
             $\Q-N$-name of the column. It is here that we make use of the fact that the partition $\PP'$ is strictly ergodic so that its topological entropy is strictly less than log a, 
             while we have a symbols at our disposal. 
             
             The encoding is done as follows.  For $N$ sufficiently large the set of possible $\PP' -(\frac{\epsilon}{3}\times N)$ names of points is at most 
             $$e^{ (h + \frac{\epsilon}{100}) \times (\frac{\epsilon}{3} \times N)}.$$
             
              Since $\Q$ is also strictly ergodic, for $N$ sufficiently large, the same estimate holds for the number of $\Q$ names of length $(\frac{\epsilon}{3}\times N)$.
              With each $\PP' -N$ name we have  a number ranging from $1$ to at most $e^{\frac{\epsilon}{100}\times N}$ which together with the $\PP'-N$ name determines 
              the entire $\Q-N$ name of that column. We then take a subset $C$ of $\{ 1,2,...a \}^{\frac{\epsilon}{3}\times N}$ that is disjoint from 
           all $\PP'-\frac{\epsilon}{3}\times N$ and $\Q-\frac{\epsilon}{3}\times N$ names and has the appropriate cardinality 
           to rename the first $\frac{\epsilon}{3}\times N$ levels  in such a way as to recover the entire partition $Q$. 
            This is done by first determining the $\PP'-\frac{\epsilon}{3}\times N$ name on these levels.
            This completes the entire $\PP'-N$ name and then  additional number  determines the $\Q-N$ name of the entire column.

               For the columns where we have no control on the further refinement by $\Q$, since they form a small set we are free to change $\PP'$ on  entire columns. 
                We simply change the partition $\PP'$ to $\Q$. In order to ensure that we also can recover the base $B$ we take another subset $D$ of 
                 $\{ 1,2,...a \}^{\frac{\epsilon}{3}\times N}$
                that is disjoint from $C$ and  all $\PP'$ and $\Q$ names and has the appropriate cardinality to relabel the first ${\frac{\epsilon}{3}\times N}$ levels 
                so that they correspond uniquely to the $\Q$ name but are recognizable. On all other levels the $\PP'$ name is simply changed ot the $\Q$ name. Finally the $\delta$
                is chosen so that the assertions made above will be satisfied for $N$ sufficiently large. 
                This completes the proof of the theorem. $\Box$

     \end{proof}


\begin{thebibliography}{WWW}

\bibitem[A]{A}
Austin, T.
     {\em Measure concentration and the weak pinsker property},
   Publ. Math. Inst. Hautes Etudes Sci., {\bf 128}, 2018, pp. 1--119.

\bibitem[AGTW]{AGTW}
Austin, T., Glasner, E.,  Thouvenot, J-P. and Weiss, B.,
     {\em An ergodic system is dominant exactly when it has positive entropy},
   Ergodic Theory Dynam. Systems, {\bf 43}, 2023, pp. 3216-3230.



\bibitem[B]{B}
 Bradley, R. C.,
      {\em Introduction to strong mixing conditions}. vol. {\bf{1}},
 Kendrick Press, Heber City, UT, 2007.

\bibitem[BKS]{BKS}
Burton, R. M. and Keane, M. S. and Serafin, Jacek,
 {\em Residuality of dynamical morphisms},
  Colloquium Mathematicum,
    {\bf{84/85}}, 2000, pp. 307-317







\bibitem[DP]{DP}
Danilenko, A. I. and Park, K. K.,
   {\em Generators and Bernoullian factors for amenable actions and
              cocycles on their orbits},
Ergodic Theory Dynam. Systems, {\bf{22}}, 2002, pp. 1715-1745.

\bibitem[D]{D}
Denker, M
     {\em On strict ergodicity},
   Math. Zeit. {\bf{134}}, 1973, pp. 231-253.




\bibitem[F]{F}
Furstenberg, H.,
    {\em Recurrence in ergodic theory and combinatorial number theory},$\\$
    M. B. Porter Lectures $\\$
     Princeton University Press, Princeton, NJ, 1981.

\bibitem[G1]{G1}
Glasner, E.,
     {\em Ergodic theory via joinings},
    Mathematical Surveys and Monographs, {\bf{101}},
  American Mathematical Society, Providence, RI, 2003.



\bibitem[G2]{G2} http://www.math.tau.ac.il/~glasner/

\bibitem[GTW1]{GTW1}
Glasner, E. and Thouvenot, J.-P. and Weiss, B.,
    {\em On some generic classes of ergodic measure preserving
              transformations},
Trans. Moscow Math. Soc., {\bf{82}} 2021, pp. 15-36.


\bibitem[GTW2]{GTW2}
Glasner, E.  Thouvenot, J.-P. and Weiss, B.,
    {\em Entropy theory without a past},
  Ergodic Theory Dynam. Systems, {\bf{20}}, 2000, pp. 1355-1370.

\bibitem[H]{H}
 Halmos, P. R.,
    {\em Lectures on ergodic theory},
   Publications of the Mathematical Society of Japan {{\bf 3}}, Tokyo, 1956.




\bibitem[K]{K}
Krieger, W.,
   {\em On entropy and generators of measure-preserving
              transformations},
   Transactions of the American Mathematical Society,
   {\bf{149}}, 1970, pp. 453--464.



\bibitem[KW]{KW}
Kifer, Y. and Weiss, B.,
     {\em Generating partitions for random transformations},
     Ergodic Theory Dynam. Systems, {\bf{22}}, 2002, pp. 1813-1830.






\bibitem[M]{M}
 Martin, N. F. G.,
  {\em The Rosenblatt mixing condition and Bernoulli shifts},
 The Annals of Probability {\bf{2}}, 1974, pp. 333-338.


\bibitem[OW]{OW}
Ornstein, D. S. and Weiss, B.,
 {\em Entropy and isomorphism theorems for actions of amenable groups},
  J. Analyse Math. {\bf{48}}, 1987, pp. 1-141.

\bibitem[P]{P}
Parreau, F. ,
{\em Facteurs disjoints des transformations m\'{e}langeantes}
arXiv:2307.01562v1

\bibitem[R]{R}
Rudolph, D. J.,
      {\em Fundamentals of measurable dynamics}, $\\$
    Oxford Science Publications,$\\$
       The Clarendon Press, Oxford University Press, New York, 1990.

\bibitem[Ro]{Ro}
Rosenthal, A.,
  {\em Finite uniform generators for ergodic, finite entropy, free
              actions of amenable groups},
   Probab. Theory Related Fields, {\bf{77}}, 1988, pp. 147-166.



\bibitem[Sc]{Sc}
  Schnurr, M. ,
      {\em Generic properties of extensions},
   Ergodic Theory and Dynamical Systems, {\bf{39}}, 2019 pp. 3144-3168.

\bibitem[Sh]{Sh}
 Shields, P. C.,
 {\em The ergodic theory of discrete sample paths},$\\$
  Graduate Studies in Mathematics {\bf{13}},$\\$
    American Mathematical Society, Providence, RI, 1996.


\bibitem[Sm]{Sm}
Smorodinsky, M. ,
   { \em A partition on a Bernoulli shift which is not weakly Bernoulli},
   Math. Systems Theory , {\bf{5}}, 1971 pp. 201-203.


\bibitem[T]{T}
Thouvenot, J-P. ,
     {\em Quelques propri\'{e}t\'{e}s des syst\`{e}mes dynamiques qui se d\'{e}composent
              en un produit de deux syst\`{e}mes dont l'un est un sch\'{e}ma de
              Bernoulli},
   Israel J. Math.{\bf{21}}, 1975, pp. 177-207.

\bibitem[W]{W}
Weiss, B.
{\em Strictly ergodic models for dynamical systems},
   Bull. Amer. Math. Soc. (N.S.) {\bf{13}}, 1985, pp. 143-146.





\end{thebibliography}
\end{document}